\documentclass[11pt, oneside]{article}   	
\usepackage[margin=1in]{geometry}                		
\usepackage{graphicx}				

\usepackage{amssymb}

\usepackage{csquotes}
\usepackage{placeins}
\usepackage{amsthm}
\usepackage{amssymb}
\usepackage{amsmath}
\usepackage{enumitem}

\usepackage{hyperref}

\usepackage[capitalize]{cleveref}

\usepackage{bm}
\renewcommand{\vec}{\bm}
\newcommand{\E}{\mathbb E}
\newcommand{\RR}{\mathbb R}
\newcommand{\Z}{\mathbb{Z}}
\newcommand{\Mat}{\operatorname{Mat}}
\newcommand{\Tr}{\operatorname{Tr}}

\newcommand{\cA}{\mathcal{A}}
\newcommand{\disc}{\operatorname{disc}}
\newcommand{\herdisc}{\operatorname{herdisc}}
\newcommand{\kgl}{\operatorname{kgl}}
\newcommand{\herkgl}{\operatorname{herkgl}}
\newcommand{\detlb}{\operatorname{detlb}}

  \newcommand{\eps}{\varepsilon}

\newtheorem{thm}{Theorem}
\numberwithin{thm}{section}
\newtheorem{conj}[thm]{Conjecture}
\newtheorem{prop}[thm]{Proposition}
\newtheorem{cor}[thm]{Corollary}
\newtheorem{claim}[thm]{Claim}
\newtheorem{obs}[thm]{Observation}
\newtheorem{lem}[thm]{Lemma}

\numberwithin{sch}{subsection}

\theoremstyle{definition}

\newtheorem{defin}{Definition}
\newtheorem{rem}[thm]{Remark}

\title{A simplified disproof of Beck's three permutations conjecture and an application to root--mean--squared discrepancy}
\author{Cole Franks\thanks{Department of Mathematics, Rutgers University, Piscataway, NJ 08854.
Supported in part by Simons Foundation award 332622.}}

\begin{document}
\maketitle
\abstract{A $k$--permutation family on $n$ vertices is a set system consisting of the intervals of $k$ permutations of the integers $1$ through $n$. The discrepancy of a set system is the minimum over all red--blue vertex colorings of the maximum difference between the number of red and blue vertices in any set in the system. In 2011, Newman and Nikolov disproved a conjecture of Beck that the discrepancy of any $3$--permutation family is at most a constant independent of $n$. Here we give a simpler proof that Newman and Nikolov's sequence of $3$--permutation families has discrepancy $\Omega(\log n)$. We also exhibit a sequence of $6$--permutation families with root--mean--squared discrepancy $\Omega(\sqrt{\log n})$; that is, in any red--blue vertex coloring, the square root of the expected difference between the number of red and blue vertices in an interval of the system is $\Omega(\sqrt{\log n})$.}
\section{Introduction}
The discrepancy of a set system is the extent to which the sets in a set system can be simultaneously split into two equal parts, or two--colored in a balanced way. Let $\cA$ be a collection (possibly with multiplicity) of subsets of a finite set $\Omega$. The discrepancy of a two--coloring $\chi: \Omega \to \{\pm 1\}$ of the set system $(\Omega, \cA)$ is the maximum imbalance in color over all sets $S$ in $\cA$. The discrepancy of $(\Omega, \cA)$ is the minimum discrepancy of any two--coloring of $\Omega$. Formally,
$$\disc(\Omega, \cA) := \min_{\chi:\Omega \to \{+1, -1\}} \max_{S \in \cA} |\chi(S)|,$$
where $\chi(S) = \sum_{x \in S} \chi(x)$.

A central goal of the study of discrepancy is to bound the discrepancy of set systems with restrictions or additional structure. Here we will be concerned with set systems constructed from permutations. A permutation $\sigma:[n] \to [n]$ determines the set system $([n], \mathcal A_\sigma)$ where $\mathcal A_\sigma = \{\{\}, \{\sigma(1) \}, \{\sigma(1),\sigma(2)\}, \dots, [n]\}$. For example, if $e:[3] \to [3]$ is the identity permutation, then $\mathcal A_e = \{\{\},\{1\}, \{1,2\}, \{1,2,3\}\}$. Equivalently, $\mathcal A_\sigma$ is a maximal chain in the poset $2^{[n]}$ ordered by inclusion. A $k$--permutation family is a set system of the form $([n], \mathcal{A}_{\sigma_1} +  \dots + \mathcal{A}_{\sigma_k})$ where $\sigma_1, \dots, \sigma_k:[n] \to [n]$ are permutations and $+$ denotes multiset sum (union with multiplicity).
By Dilworth's theorem, the maximal discrepancy of a $k$-permutation family is the same as the maximal discrepancy of a set system of \emph{width} $k$, that is, one that contains no antichain of cardinality more than $k$.

It is easy to see that a $1$--permutation family has discrepancy at most $1$, and the same is true for $2$--permutation families \cite{Sp87}. Beck conjectured that the discrepancy of a $3$--permutation family is $O(1)$. More generally, Spencer, Srinivasan and Tetali conjectured that the discrepancy of a $k$-permutation family is $O(\sqrt{k})$ \cite{SST01}. Both conjectures were recently disproven by Newman and Nikolov \cite{NN11}. They showed the following:
\begin{thm}[\cite{NN11}]\label{thm:nn} There is a sequence of $3$--permutation families on $n$ vertices with discrepancy $\Omega(\log n)$.
\end{thm}
Along with Neiman, in \cite{NNN12} they showed this implies that a natural class of rounding schemes for the Gilmore--Gomory linear programming relaxation of bin--packing, such as the scheme used in the Kamarkar-Karp algorithm, incur logarithmic error.

Spencer, Srinivasan and Tetali proved an upper bound that matches the lower bound of Newman and Nikolov for $k = 3$.
\begin{thm}[\cite{SST01}]\label{thm:sst} The discrepancy of a $k$--permutation family on $n$ vertices is $O(\sqrt{k} \log n).$
\end{thm}
They showed that the upper bound is tight for for $k \geq n$. However, it is open whether this upper bound is tight for $3 < k = o(n)$. In fact, no one has exhibited a family with discrepancy $\omega(\log n)$ and $k = o(n)$, let alone proved lower bounds with logarithmic dependency on $n$ that tend to $\infty$ as a function of $k$.
\begin{conj}\label{conj:growing_k} The discrepancy of a $k$--permutation family on $n$ vertices is 
$$\Omega(f(k)\log n)$$
where $f(k) = \omega(1)$.
\end{conj}
In this paper, we present a new analysis of the counterexample due to Newman and Nikolov. We replace their complicated case analysis by a simple argument using norms of matrices, albeit achieving a slightly worse constant ($\log_3 n/4 \sqrt{2}$ vs their $\log_3 n/3$). Our analysis generalizes well to larger permutation families, and can hopefully be extended to handle \cref{conj:growing_k}. Our analysis also yields a new result for the root--mean--squared discrepancy, defined as 
$$\disc_2(\Omega, \cA) = \min_{\chi:[n] \to \{\pm 1\}}\sqrt{\frac{1}{|\cA|} \sum_{S \in \cA} |\chi(S)|^2}.$$
Define the \emph{hereditary root--mean--squared discrepancy} by $\herdisc_2(\Omega, \cA) = \max_{\Gamma \subset \Omega} \disc_2(\Gamma, \cA|_{\Gamma})$.
\begin{thm}\label{thm:rms}There is a sequence of $6$--permutation families on $n$ vertices with root--mean--squared discrepancy $\Omega(\sqrt{\log n})$.
\end{thm}
For $k = 6$, \cref{thm:rms} matches the upper bound of $\sqrt{k\log n}$ for the root--mean--squared discrepancy implied by the proof of \cref{thm:sst} in \cite{SST01}. Further, the lower bound implied by \cite{KGL18} for the hereditary root--mean--squared discrepancy is constant for families of constantly many permutations. This fact was communicated to the author by Aleksandar Nikolov; we provide a proof in \cref{sec:rms} for completeness. The lower bound in \cite{KGL18} is smaller than $\herdisc_2(\Omega, \cA)$ by a factor of at most $\sqrt{\log n}$, so \cref{thm:rms} shows that the $\sqrt{\log n}$ gap between $\herdisc_2(\Omega, \cA)$ and the lower bound in \cite{KGL18} is best possible.


\section{The set system of Newman and Nikolov}

Our proof of \cref{thm:nn} uses the same set systems as Newman and Nikolov. For completeness, we define and slightly generalize the system here. The vertices of the system will be $r$--ary strings, or elements of $[r]^d$. For Newman and Nikolov's set system, $r = 3$. Bold letters, e.g. $\vec a$, denote strings in $[r]^d$. If $\vec a = a_1\dots a_d \in [r]^d$ is a string, for $0 \leq k \leq d$ let $\vec a [k]$ denote the string $a_1\dots a_k$, with $\vec a[0]$ defined to be the empty string $\vec \eps$. Here $[r]^0$ denotes the set containing only $\vec \eps$. If $\vec a \in [r]^{d_1}$ and $\vec b \in [r]^{d_2}$ are strings, their concatentation in $[r]^{d_1 + d_2}$ is denoted $\vec a \vec b$. $\tau$ will denote the permutation of $[r]$ given by $\tau(i) = r - i + 1$, the permutation reversing the ordering on $[r]$. If $j \in \vec [r]$, then $\overline j$ denotes the all $j$'s string of length $d$; e.g. $\overline 3:= \underbrace{33..3}_{d}$.
\begin{defin}[The set system $(\text{[}r\text{]}^d, \cA_P)$]
Given a permutation $\sigma$ of $[r]$, we define a permutation of $[r]^d$ by acting digitwise by $\sigma$. Namely, $\sigma \cdot \vec{a} := \sigma (a_1)\sigma(a_2) \dots \sigma(a_d)$. Let $\cA_P $ consist of the sets
$$ E_{\sigma, \vec a} := \{\vec{b}: \sigma \cdot \vec{b} < \sigma \cdot \vec{a}\}$$
for $\sigma \in P, \vec a \in [r]^d,$ where $<$ is the lexicographic ordering on $[r]^d$. The union of $\cA_P $ and $\{[r]^d\}$, denoted $\cA_P^+$, is the $|P|$-permutation family defined by the permutations $\vec a \mapsto \sigma^{-1} \cdot \vec a$ for $\sigma \in P$. 
\end{defin}
The system of Newman and Nikolov is $([3]^d, \cA_C^+)$ with $C = \{e, (1,2,3), (1, 3, 2)\}$. That is, $C$ is the cyclic permutations of $3$. We first bound the discrepancy of the larger set system $([3]^d, \cA_{S_3})$, which is contained in a $6$--permutation family.


\begin{prop}\label{disc} If $r\geq 3$ is odd, then
$$\disc([r]^d, \cA_{S_r})\geq \frac{d}{2\sqrt{2}}.$$
\end{prop}
Before we prove \cref{disc}, we show how it implies \cref{thm:nn}, which follows immediately from the next corollary. This corollary is specific to $r = 3$.
\begin{cor}\label{cor:nn}
$\disc([3]^d, \cA_C^+) \geq \left\lfloor \frac{d}{4 \sqrt{2}}\right\rfloor.$
\end{cor}
\begin{proof}[Proof of \cref{cor:nn}]
Consider a coloring $\chi$ of $[3]^d$. If $|\chi([3]^d)| \geq d/4 \sqrt{2}$, then the discrepancy of $\chi$ is at least $d/4 \sqrt{2}$ because $[3]^d \in \cA_C^+$. If $|\chi([3]^d)| \leq d/4 \sqrt{2}$, then for all $\sigma \in S_3$ and all $\vec{a}$, we have $|\chi(E_{\sigma, \vec{a}}) + \chi(E_{\tau \circ \sigma, \vec{a}})| \leq d/4 \sqrt{2}+1$. This is because the vertices will be in the reverse order under the action of $\sigma$ and $\tau \circ \sigma $ and so $E_{ \tau \circ\sigma , \vec{a}} = [3]^d \setminus (E_{\sigma, \vec{a}} \cup \{\vec a\}).$ By Claim \ref{disc}, there is a $\sigma \in S_3$ and a string $\vec{a}$ such that $|\chi(E_{\sigma, \vec{a}})| \geq d/2 \sqrt{2},$ so $|\chi(E_{\tau \circ \sigma, \vec{a}})| \geq d/4 \sqrt{2} -1$. One of $\sigma$ and $\tau \circ \sigma $ is cyclic, i.e. in $C$, so $\disc([3]^d, \cA_{S_3}) \geq \lceil d/4 \sqrt{2} - 1 \rceil = \lfloor d/4 \sqrt{2}  \rfloor$. \end{proof}
\begin{proof}[Proof of \cref{disc}]
In order to show the discrepancy of $\disc([r]^d, \cA_{S_r})$ is at least $K$, it is enough to show that given a coloring $\chi:[r]^d \to \{\pm 1\}$, or even $\chi:[r]^d \to  2 \Z - 1,$ we can choose $\sigma \in S_r$ and $\vec a \in [r]^d$ so that $|\chi(E_{\sigma , \vec a})|$ is at least $K$. Recall that a seminorm is a function satisfying all the properties of a norm other than nondegeneracy. Given $\chi$, we define an $r\times r$ matrix $M_\chi(\vec a)$ and a seminorm $\| \cdot \|_{S_r}$ on matrices such that for all $\vec a$, 
$$\max_{\sigma}|\chi(E_{\sigma , \vec a})| = \|M_\chi(\vec a)\|_{S_r}.$$ The advantage is that we only need to show how to choose $\vec a \in [r]^d$ to maximize $\|M_\chi(\vec a)\|_{S_r}$. 

\begin{defin}[The seminorm $\| \cdot \|_{S_r}$]
For $M \in \Mat_{r\times r}(\RR)$, and $\sigma \in S_r$, define $\sigma \cdot M:= \sum_{i,j \in [r],\; \sigma(i) > \sigma(j)} M_{i,j}$. Now let 
$$\| M \|_{S_r} =  \max_{\sigma \in S_r} \left|\sigma \cdot M\right|.$$
\end{defin}
\begin{rem}\label{rem:seminorm}This seminorm is well--studied; if $M$ is the $0,1$ adjacency matrix of a directed graph $G$, then $\| M \|_{S_r}$ is the maximum size of an acyclic subgraph of $G$. In \cite{GMR08} it is shown that, assuming the unique games conjecture, $\|M\|_{S_r}$ is $\mathbf{NP}$--hard to approximate even for $M$ antisymmetric.
\end{rem}
\begin{defin}[The matrix $M_\chi(\vec a)$]\label{dfn:matrix} For odd $r$, the set system $([r]^d, \cA_P )$ is suitable for induction. In particular, given $\chi:[r]^d \to 2 \Z - 1$ we may extend it to a coloring 
$$\chi: [r]^0 \cup [r]^1 \cup \dots \cup [r]^d \to 2 \Z - 1$$ by inductively defining $\chi(\vec b) = \sum_{i \in [r]} \chi(\vec b i)$ for $0\leq k \leq d$, $\vec b \in [r]^k$. We recursively define a sequence of matrices taking strings as parameters by defining $M() = 0$ and for $\vec a \in [r]^{k}$ defining
$$M_{\chi}(\vec a) = M_{\chi}(\vec a[k-1]) + L_\chi(\vec a),$$
where $L_\chi(\vec a)$ is a very simple matrix. $L_\chi(\vec a)$ has only the $a_k$'th row nonzero, and the contents of this row are $\chi(\vec a[k-1]1), \chi(\vec a[k-1]2) \dots, \chi(\vec a[k-1] r)$. Equivalently, $L_\chi(\vec a)_{i,j} = \delta_{i, a_d} \chi(\vec a[d-1] j)$.
\end{defin}
We now prove that this matrix and seminorm have the promised property. 
\begin{claim}\label{claim:seminorm} For all $\chi:[r]^d \to 2 \Z - 1$, $\max_{\sigma}|\chi(E_{\sigma , \vec a})| = \|M_\chi(\vec a)\|_{S_r}.$ In particular, 
\begin{align}\chi(E_{\sigma , \vec a}) = \sigma \cdot M_\chi(\vec a).\label{eq:matrix}\end{align}
\end{claim}
\begin{proof}[Proof of \cref{claim:seminorm}]
The proof is by induction. If $d = 0$, this is trivially true, because $[r]^0$ consists only of the empty string $\vec \eps$ and $E_{\sigma , \vec \eps} = \emptyset,$ so both sides of \cref{eq:matrix} are zero. Now suppose the claim is true for $d-1 \geq 0$. In evaluating $\chi(E_{\sigma , \vec a})$, we can split the sum into the terms for which $\sigma(b_{k}) < \sigma(a_{k})$ for some $k < d$, and those for which $b_{k} = a_{k}$ for all $k < d$:
\begin{align*}
\chi(E_{\sigma , \vec a}) &= \sum_{\vec b \in [r]^d,\; \sigma \cdot \vec b < \sigma \cdot \vec a} \chi(\vec b)\\
&= \sum_{\vec b \in [r]^{d}:\; \sigma \cdot  \vec b[d-1] < \sigma \cdot \vec a [d-1]} \chi(\vec b) + \sum_{j \in [r],\; \sigma(j) < \sigma(a_d)} \chi(\vec a[d-1] j)
\end{align*}
The first term is precisely $\chi (E_{\sigma, \vec a[d-1]})$ for the coloring $\chi:[r]^{d-1} \to 2\Z-1$ induced by $\chi$ as in \cref{dfn:matrix}. The second term is, by definition, $\sigma \cdot L_\chi(\vec a).$ Applying the induction hypothesis to the first term, we obtain 
\begin{align*}
\chi(E_{\sigma , \vec a}) =   \sigma \cdot \left( M_\chi(\vec a[d-1]) + L_\chi(\vec a)\right) = \sigma \cdot M_\chi(\vec a)_{i,j}.
\end{align*}
\end{proof}
Now that we have \cref{claim:seminorm}, it remains to bound $\min_\chi \max_{\vec a} \|M_\chi(\vec a)\|_{S_r}$ below. This quantity is at least the value of the following $d$-round game played between a ``minimizer" and a ``maximizer". The states of the game are a $r\times r$ integer matrix $M$, and the value is the maximium value of $\|M\|_{S_3}$ at the end of the game. The matrix $M$ is updated in each round as follows.
\begin{enumerate} 
\item The minimizer chooses a vector $v$ in $(2 \Z - 1)^r$; that is, a list of $r$ odd numbers.
\item The maximizer chooses a number $i \in [r]$ and adds $v$ to the $i^{th}$ row of $M$.
\end{enumerate}

The coloring $\chi:[r]^d \to (2 \Z - 1)$ determines the following strategy for the minimizer: if the maximizer chose rows $a_1, \dots, a_{k-1}$ in rounds $1, \dots, k-1$, the minimizer chooses the vector $v = \chi(a_1 \dots a_{k-1}1), \dots, \chi(a_1 \dots a_{k-1} r)$ in round $k$, where $\chi$ on $[r]^k$ is determined by $\chi$ on $[r]^d$ as in \cref{dfn:matrix}. If the minimizer plays this strategy and the maximizer plays $\vec a$, the matrix after the $k^{th}$ round will be $M_\chi(\vec a[k])$,  because $L_\chi(\vec a [k])$ has $v$ in the $a_k^{th}$ row and zeroes elsewhere. If the minimizer is constrained to choose $w,v$ in the $(k-1)^{st}$ and $k^{th}$ rounds, respectively, such that $\sum_{i = 1}^r v_i = w_{a_{k-1}}$, then without loss of generality the strategy of the minimizer comes from some coloring $\chi$ as above. However, the value of the game is $\Omega(d)$ even without this constraint on the minimizer.

To show this, we first bound the seminorm below by a simpler quantity. Recall that the Frobenius norm $\|M\|_F$ of a matrix $M$ is the square root of the sum of squares of its entries.
\begin{lem}\label{lem:frob}  For $\sigma \in S_r$ chosen uniformly at random, 
$$\| M \|_{S_r} \geq \sqrt{\E_\sigma  (\sigma \cdot M)^2} \geq \frac{1}{2\sqrt{2}} \| M - M^t \|_{F}.$$

\end{lem}
\begin{proof}[Proof of \cref{lem:frob}] 
The first inequality is immediate. Let $J$ be the all--ones matrix. For the second inequality, we use the identity \begin{align} \E_\sigma (\sigma \cdot M)^2 =\frac{1}{4} (\Tr  J(M + M^t) ^2) + \frac{1}{4} \E_\sigma (\sigma \cdot (M - M^t) )^2.\label{eq:expectation}
\end{align}
\cref{eq:expectation} follows because the expectation of the square of a random variable is its mean squared plus its variance, and $\E_\sigma \sigma \cdot M = \frac{1}{2} \Tr J (M + M^t)$. The second term is the variance because $M = \frac{1}{2} ( M + M^t) + \frac{1}{2}( M - M^t)$ and for any $\sigma \in S_r$, we have $\sigma \cdot \frac{1}{2} (M + M^t) = \E_\sigma \sigma \cdot M$.   

Set $A = M - M^t$. In particular, $A$ is antisymmetric.
Write \begin{align}
\E_{\sigma} (\sigma \cdot A)^2 &= \sum_{i\neq j, k \neq l} A_{i,j} A_{k,l} \E[1_{\sigma (i) > \sigma (j)} 1_{\sigma (k) > \sigma (l)}] \nonumber\\
&= \frac{1}{4}\sum_{ k \notin \{i, j\}, l \notin \{i, j\}, i\neq j}A_{i,j} A_{k,l}\label{eq:vanish_1}\\
&+ \frac{1}{3}\sum_{  k\notin\{i,j\}, i\neq j} (A_{i,j} A_{j,k} + A_{i,j} A_{k,i} + A_{i,j} A_{k,j} + A_{i,j} A_{i,k})\label{eq:vanish_2}\\
&+  \frac{1}{2}\sum_{i\neq j}A_{i,j} A_{i,j}.\nonumber
 \end{align}


This expression is obtained by computing $\E[1_{\sigma (i) > \sigma (j)} 1_{\sigma (k) > \sigma (l)}]$ in each of the four possible cases. If $|\{i,j\} \cap \{k,l\}| = 0$, then $\E[1_{\sigma (i) > \sigma (j)} 1_{\sigma (k) > \sigma (l)}] = 1/4$. If $|\{i,j\} \cap \{k,l\}| = 1$, then $\E[1_{\sigma(i) > \sigma(j)} 1_{\sigma(k) > \sigma(l)}] = 1/3$. If $i = k, j = l$, then $\E[1_{\sigma(i) > \sigma(j)} 1_{\sigma(k) > \sigma(l)}] = 1/2$. If $i = l, j = k$, then $1_{\sigma(i) > \sigma(j)} 1_{\sigma(k) > \sigma(l)} = 0$. Because $A$ is antisymmetric, \cref{eq:vanish_1} = \cref{eq:vanish_2} = 0, and so 
\begin{align}
\E_{\sigma} (\sigma \cdot A)^2 = \frac{1}{2}\| A \|_F^2\label{eq:frob_final}
\end{align}
for any antisymmetric matrix $A$. Combining \cref{eq:frob_final} and \cref{eq:expectation} completes the proof.\end{proof}







By \cref{lem:frob}, it suffices to exhibit a strategy for the maximizer that enforces $\| M - M^t\|_F \geq d$ after $d$ rounds. This is rather easy -- we may accomplish this by focusing only on two entries of $M$: the maximizer only tries to control the $1,r$ and $2,r$ entries. If in the $k^{th}$ round, minimizer chooses $v$ with $v_r > 0$, the maximizer sets $a_k = 1$. Else, maximizer sets $a_k  = 2$. Crucially, the entries of $v$ are odd numbers; in particular, they are greater than $1$ in absolute value. Further, all but the first and second rows of $M$ are zero throughout the game. Thus, in the $d^{th}$ round, $|(M - M^t)_{2,r}| + |(M - M^t)_{1,r}| \geq d$, so $\|M - M^t\|_F \geq d.$  \end{proof}
\begin{rem} To prove \cref{conj:growing_k}, it suffices to show the maximizer can achieve $\|M\|_{S_r} = f(r) d$ where $f(r) = \omega(\log r)$. A promising strategy is to replace $\| \cdot \|_{S_r}$ by another seminorm $\| \cdot \|_{*}$ and show that the maximizer can enforce $\| \cdot \|_* \geq f(r) \| Id\|_{S_r \to *}$, where $Id$ is the identity map on $\Mat_{r\times r}(\RR)$. Obvious candidates such as $\|M - M^t\|_F$ and $\|M - M^t\|_1$ do not suffice. Here $\| B\|_1$ is the sum of the absolute values of entries of $B$. For instance, the minimizer can enforce $\| M - M^t\|_F = O(d)$ or $\| M - M^t\|_1 = O(\sqrt{r}d)$, and even antisymmetric matrices $A$ can achieve $ \| A\|_{S_r} \leq \|A\|_F$ and $\| A\|_{S_r} \leq \frac{\sqrt{\log r}}{\sqrt{r}}\|A\|_1$. The first inequality is very easy to achieve, and a result of Erdos and Moon shows the second is achieved by random $\pm 1$ antisymmetric matrices \cite{EM65}. By the inapproximability result mentioned in \cref{rem:seminorm}, it is not likely that any of the easy--to--compute norms $\| \cdot \|_*$ have both $\| Id\|_{S_r \to *}$ and $\|Id\|_{* \to S_r}$ bounded by constants independent of $r$. A candidate seminorm is the cut--norm of the top--right $1/3 r \times 2/3 r$ submatrix of $M$: it is not hard to see that this seminorm is a lower bound for $\| M \|_{S_r}$.
\end{rem}

\section*{Root--mean--squared discrepancy}\label{sec:rms}
This section is concerned with the proof of \cref{thm:rms}. Before the proof, we discuss the relationship between \cref{thm:rms} and the previous lower bounds in \cite{Mat11}, \cite{NTZ12}, \cite{KGL18}, presented below. The original lower bound was for the usual $\ell_\infty$ discrepancy.
\begin{thm}[\cite{LSV86}]\label{thm:lsv} Denote by $A$ be the $|\Omega| \times |\cA|$ incidence matrix of $(\Omega, \cA)$, and define $$\detlb(\Omega, \cA) = \max_{k}\max_{B} |\det (B)|^{1/k}.$$
where $B$ runs over all $k\times k$ submatrices of $A$. Then $\herdisc_\infty(\Omega, \cA):=\max_{\Gamma \subset \Omega} \disc(\Omega, \cA) \geq \detlb(\Omega, \cA)$.
\end{thm}
It was proved in \cite{Mat11} that this lower bound behaves well under unions:
\begin{thm}[\cite{Mat11}]\label{thm:mat}
$$\detlb(\Omega, \cA_1 + \dots + \cA_k) = O\left(\sqrt{k} \;\max_{i \in [k]}\; \detlb(\Omega, \cA_i)\right),$$ where $+$ denotes the multiset sum (union with multiplicity).

\end{thm}

Next, consider the analogue of the determinant lower bound for $\disc_2$.
\begin{thm}[Theorem 6 of \cite{KGL18}; corollary of Theorem 11 of \cite{NTZ12} up to constants]\label{thm:detlb2}
Denote by $A$ be the $|\Omega| \times |\cA|$ incidence matrix of $(\Omega, \cA)$, and define 
$$\detlb_2(\Omega, \cA) = \max_{\Gamma \subset \Omega} \sqrt{\frac{m |\Gamma|}{8\pi e}} \det (A|_S^T A|_S)^{\frac{1}{2 |\Gamma|}}. $$
Then $\herdisc_2(\Omega, \cA) \geq \detlb_2(\Omega, \cA)$.
\end{thm}

\begin{thm}[Consequence of the proof of Theorem 7 of \cite{KGL18}]\label{thm:kgl} $$\herdisc_2(\Omega, \cA) =  O(\sqrt{\log n} \detlb_2(\Omega, \cA)).$$
\end{thm}
The main point of \cref{thm:kgl} and \cref{thm:detlb2} is that $\detlb_2$ is a $\sqrt{\log n}$ approximation to $\herdisc_2$. Taken together with \cref{thm:mat}, we obtain the following bound.
\begin{obs} \label{obs:2infty}
$\herdisc_2(\Omega, \cA_1 + \dots + \cA_k) = O\left( \sqrt{k \log n} \;\max_{i \in [k]} \;\herdisc_\infty(\Omega, \cA_i)\right)$

\end{obs}
\begin{proof} Applying the Cauchy-Binet identity to $\det (A^T A)$ immediately implies $\detlb_2(\Omega, \cA) = O(\detlb(\Omega, \cA)).$ By \cref{thm:kgl}, \cref{thm:mat}, and \cref{thm:lsv},
\begin{align*}\herdisc_2(\Omega, \cA_1 + \dots + \cA_k) &= O\left(\sqrt{\log n}\detlb_2(\Omega, \cA_1 + \dots + \cA_k)\right)\\
& = O\left(\sqrt{\log n}\detlb(\Omega, \cA_1 + \dots + \cA_k)\right)\\
& = O\left(\sqrt{k\log n} \;\max_{i \in [k]}\; \detlb(\Omega, \cA_i)\right).\\
& = O\left( \sqrt{k \log n} \;\max_{i \in [k]} \;\herdisc_\infty(\Omega, \cA_i)\right)
\end{align*}
\end{proof}

If $(\Omega, \cA)$ is a $1$--permutation family, then $\herdisc_\infty(\Omega, \cA) = 1$. Combined with \cref{obs:2infty}, we immediately recover the bound from \cite{Sp87}.

\begin{cor}\label{cor:disc2} If $(\Omega, \cA)$ is a $k$--permutation family, then $\herdisc_2(\Omega, \cA) \leq \sqrt{k \log n}.$
\end{cor}
\cref{thm:rms} implies that, for constant $k$, \cref{cor:disc2} and \cref{obs:2infty} are tight. Further, the reasoning for \cref{obs:2infty} shows that for $k$ constant, $\detlb_2(\Omega, \cA)$ is constant for $k$--permutation families $(\Omega, \cA)$. Thus, \cref{thm:rms} shows that \cref{thm:kgl} is best possible in the sense that there can be a $\Omega(\sqrt{\log n})$ gap between $\detlb_2(\Omega, \cA)$ and $\herdisc_2(\Omega, \cA)$. 

We now proceed with the proof of \cref{thm:rms}, which follows immediately from the below proposition.
\begin{prop}
$$\disc_2([3]^d, \cA_{S_3}) = \Omega(\sqrt{d}).$$
\end{prop}

\begin{proof} We must show that for every $\chi:[3]^d \to \{\pm 1\}$, $\E_{\sigma, \vec a} [|\chi(E_{\sigma, \vec a})|^2] = \Omega(\sqrt{d})$. By \cref{lem:frob} and \cref{eq:matrix}, 
\begin{align} \E_{\sigma, \vec a} [|\chi(E_{\sigma, \vec a})|^2] \geq \frac{1}{2 \sqrt{2}}\E_{\vec{a}} \|M_\chi(\vec{a}) - M_\chi(\vec{a})^t\|_F^2. \label{eq:more_frob}\end{align}
Consider again the game from the proof of \cref{disc}. If the minimizer plays according to $\chi$ and the maximizer chooses rows randomly, then the outcome of the game will be the matrix $M_\chi(\vec{a})$ for $\vec a$ chosen uniformly at random. Let $M_i$ be the matrix in the $i^{th}$ round of the game. It is enough to show that $\E \| M_d - M_d^t\|_F^2 = \Omega(d)$ if the minimizer plays according to some coloring $\chi$ and the maximizer chooses rows randomly. Consider the sequence of random variables $Y_i = (M_i - M_i^t)_{1,2} + (M_i - M_i^t)_{2,3} - (M_i - M_i^t)_{1,3}$.  It is enough to show that $\E[ Y_d^2]$ is large, because by the Cauchy-Schwarz inequality we have 
\begin{align}\| M_d - M_d^t\|_F^2 \geq |Y_d|^2/3.\label{eq:martin_bound}\end{align}
 The sequence $Y_i$ is a martingale with respect to $M_i$, because $Y_i - Y_{i-1}| M_{i-1}$ is equally likely to be $v_2 - v_3$, $v_3 - v_1$, or $v_1 - v_2$ if the minimizer chooses $v$ in round $i$. Because $Y_i$ is a martingale, 
\begin{align}\E Y_d^2 =  \sum_{i = 1}^d  \E_{M_{i-1}} \left[\left.\frac{(v_2 - v_3)^2 + (v_1 - v_3)^2 + (v_1 - v_2)^2}{3} \right| M_{i -1}\right].\label{eq:variance}\end{align}
There are strategies that make the above quantity small, but they are bad strategies if they come from a coloring $\chi$. Strategies induced by $\chi$ satisfy that $v_1 + v_2 + v_3=w_{a_{k-1}}$ if the minimizer chose $w,v$ in round $k-1,k$, respectively. If $v_1,v_2,v_3$ are typically equal, then intuitively the entries of $w_i$ are growing, which would lead to bad colorings. We now make this intuition precise. First we show that if $\chi()$ is very large, the root--mean--squared discrepancy is high. Recall that $\chi() = \sum_{\vec a \in [3]^d} \chi(\vec a) = \chi(E_{\overline 3, e}) \pm 1$, where $e$ is the identity permutation. Considering the contribution from only the edge $E_{\overline 3, e}$ yields the following trivial lemma, which allows us to assume that $|\chi()| \leq 1.9^d$.
\begin{lem}\label{lem:big_entry} If $|\chi()| \geq 1.9^d$, then $\E_{\vec a, \sigma} [ \chi (E_{\vec a, \sigma})^2] = \Omega\left(\left(\frac{1.9^2}{3}\right)^d\right) = \omega(d)$. 
\end{lem}

 Next, we show that this implies many cancellations, and that this implies \cref{eq:variance} is large. For $\vec a \in [r]^0 \cup [r]^1 \cup \dots \cup [r]^d$, define the \emph{cancellation} of $\chi$ at $\vec a$ by 
\begin{align}C_\chi(\vec a) = \sum_{i \in [3]} |\chi (\vec a i)| - \left|  \chi (\vec a)\right|.\label{eq:cancellation}\end{align}
For $i \in [d]$, define the average cancellation $\overline{C^i_\chi}= \E_{\vec a \in [r]^i} C_\chi(\vec a)$. If the strategy of the minimizer has high cancellation, then the strategy also has \cref{eq:variance} large: if the minimizer is following the strategy induced by $\chi$, in response to $\vec a = a_1\dots a_{k-1}$ he plays the vector $v = (\chi(\vec a 1), \chi(\vec a 2), \chi(\vec a 3))$. Then 
\begin{align} C_\chi(\vec a)^2 &=  \left( |v_1| + |v_2| + |v_3| - | v_1 + v_2 + v_3|\right)^2\\
 &\leq  \left(|v_1 - v_2| + |v_2 - v_3| + |v_3 - v_1|\right)^2\\
  &\leq 3\left( |v_1 - v_2|^2 + |v_2 - v_3|^2 + |v_3 - v_1|^2\right).\end{align}
  Thus, if the strategy of the minimizer is induced by $\chi$, we have 
  \begin{align}
  \sum_{i = 1}^d \E_{M_{i-1}} \left[\left.\frac{(v_2 - v_3)^2 + (v_1 - v_3)^2 + (v_1 - v_2)^2}{3} \right| M_{i -1}\right]  &\geq \sum_{i = 1}^d \E_{\vec a \in [r]^i} \left[C_\chi(\vec a)^2\right]\\
  & \geq \sum_{i = 1}^d \overline{C_\chi^i}^2 \\
  & \geq \frac{1}{d}\left(\sum_{i = 1}^d \overline{C_\chi^i}\right)^2.\label{eq:cancel} 
  \end{align}
The next lemma shows that $\sum_{i = 1}^d \overline{C_\chi^i}$ is large enough.
\begin{lem}\label{lem:cancel} If $|\chi()| \leq 1.9^d$, then 
$\sum_{i = 1}^d \overline{C^i_\chi} = \Omega(d).$
\end{lem}
\begin{proof}[Proof of \cref{lem:cancel}] Define the average absolute value $\overline{|\chi_i|} = \E_{\vec a \in [r]^i} \overline{|\chi(\vec a)|}$. Note that $\overline{|\chi_i|} \geq 1$. Taking the expectation of both sides of \cref{eq:cancellation} yields the identity
$$\overline{C^i_\chi} = 3 \overline{|\chi_{i+1}|} - \overline{|\chi_{i}|}.$$
Summing over $i \in [j]$ gives
$$ \sum_{i = 1}^j \overline{C^i_\chi} = - \overline{|\chi_j|} + \overline{|\chi_0|} +   2 \sum_{i = 1}^{j-1} \overline{|\chi_{i}|}.$$
Finally, there exists $j \in \{\lfloor.01d\rfloor, \dots,  d\}$ such that $\overline{|\chi_j|} \leq 2 \overline{|\chi_{j-1}|}$, else $|\chi()| = \overline{|\chi_d|}\geq 2^{.99 d } > 1.9^d$. With this $j$, 
$$ \sum_{i = 1}^j \overline{C^i_\chi} \geq 2 \sum_{i = 1}^{j-2} \overline{|\chi_{i}|} = \Omega(d).$$
\end{proof}
Finally, using \cref{lem:big_entry} to apply \cref{lem:cancel} and then combining \cref{eq:cancel} with \cref{eq:variance}, \cref{eq:martin_bound} and \cref{eq:more_frob} yields  $\E_{\sigma, \vec a} [|\chi(E_{\sigma, \vec a})|^2] = \Omega(d)$, completing the proof. \end{proof}

\section{Another inequality for root--mean--squared discrepancy}
The proof of \cref{thm:kgl} proceeds through an intermediate quantity whose definition we recall now. Denote by $A$ be the $|\Omega| \times |\cA|$ incidence matrix of $(\Omega, \cA)$, and let $\lambda_l$ be the $l^{th}$ largest eigenvalue of $A^TA$. Define 
\begin{align*} \kgl(\Omega, \cA) &= \max_{1 \leq l \leq \min\{|\Omega|, |\cA|\}} \frac{l}{e} \sqrt{\frac{\lambda_l}{8 \pi |\Omega||\cA|}}\\
\textrm{and }  \herkgl(\Omega, \cA) &= \max_{\Gamma \subset \Omega} \kgl(\Gamma, \cA|_{\Gamma}).
\end{align*}

\begin{thm}[Corollary 2 and consequence of the proof of Theorem 7 of \cite{KGL18}]\label{thm:herkgl} Then $$\herkgl(\Omega, \cA) \leq \detlb_2(\Omega, \cA) \leq \herdisc_2(\Omega, \cA)= O( \sqrt{\log n} \herkgl(\Omega, \cA)).$$
\end{thm}
Like $\detlb$, the quantity $\herkgl$ behaves nicely under unions. 
\begin{obs}\label{obs:union} 
$\herkgl(\Omega, \cA_1 + \dots + \cA_k) \leq k \max_{i \in [k]} \; \herkgl(\Omega, \cA_i).$
\end{obs}

\begin{proof}[Proof of \cref{obs:union}] Let $C = \max_{i \in [k]} \; \herkgl(\Omega, \cA_i)$. It is enough to show $\kgl(\Gamma, (\cA_1 + \dots + \cA_k)|_\Gamma) \leq kC$ for any $\Gamma \subset \Omega$. Let $|\Gamma| = n$, $m_i = |\cA_i|$, and $\sum m_i = m$. If $A_i$ is the incidence matrix of $(\Gamma, \cA_i|_{\Gamma})$ and $A$ that of $(\Gamma, (\cA_1 + \dots + \cA_k)|_\Gamma)$, then 
$$A^TA = A_i^T A_i + \dots + A_i^T A_i.$$
Weyl's inequality on the eigenvalues of Hermitian matrices asserts that if $H_1$ and $H_2$ are $n\times n$ Hermitian matrices then $\lambda_{i + j - 1}(H_1+ H_2) \leq \lambda(H_1)_{i} + \lambda(H_2)_{j}$ for all $1 \leq i, j \leq i + j -1 \leq n$. Applying this inequality inductively, $\lambda_l(A^TA) \leq \sum_{i = 1}^k \lambda_{\lceil l/k \rceil} (A_i^T A_i)$. Thus, 
\begin{align*}
\kgl(\Gamma, (\cA_1 + \dots + \cA_k)|_\Gamma) &= \max_{1 \leq l \leq \min\{n, m\}} \frac{l}{e} \sqrt{ \frac{\lambda_l(A^T A)}{8 \pi mn}} \\
&\leq \max_{1 \leq l \leq \min\{n, mk\}} \frac{l}{e} \sqrt{ \frac{\sum_{i = 1}^k \lambda_{\lceil l/k \rceil} (A_i^T A_i)}{8 \pi mn}}\\
&\leq k C.
\end{align*}
where in the last line we used $\sum m_i = m$ and $\lambda_{\lceil l/k \rceil}(A_i^TA_i) \leq 8 \pi m_in \left( \frac{C e k}{l} \right)^2$ from our assumption that
 $\kgl(\Gamma, \cA_i|_{\Gamma}) \leq \herkgl(\Omega, \cA_i) \leq C$.
\end{proof}

A pleasant consequence of \cref{obs:union} and \cref{thm:herkgl} is a variant of \cref{obs:2infty}.
\begin{cor} \label{cor:union}


$\herdisc_2(\Omega, \cA_1 + \dots + \cA_k) = O\left( k \sqrt{ \log n} \;\max_{i \in [k]} \;\herdisc_2(\Omega, \cA_i)\right).$
\end{cor}
Improving $k$ to $\sqrt{k}$ in \cref{obs:union}, and as a consequence, in \cref{cor:union}, would strengthen \cref{obs:2infty}.

\section*{Acknowledgements}
The author would like to thank Michael Saks and Aleksandar Nikolov for many interesting discussions. The author also thanks Aleksandar Nikolov and Shachar Lovett for suggesting the application of this argument to the root--mean--squared discrepancy, and especially to Aleksandar Nikolov for communicating \cref{obs:2infty} and suggesting the connection with \cite{KGL18}, \cite{NTZ12}, and \cite{Mat11}.

\bibliographystyle{plain}
\bibliography{/Users/Cole/RMD_arxiv/discrepancy_arxiv.bib}

\end{document}